\documentclass[11pt,reqno]{amsart}
\usepackage[a4paper,margin=27mm]{geometry}
\usepackage[T1]{fontenc}
\usepackage{lmodern}
\usepackage{amsmath,amssymb,amsthm,mathtools}
\usepackage{microtype}
\usepackage[hidelinks]{hyperref}
\numberwithin{equation}{section}
\newtheorem{theorem}{Theorem}[section]
\newtheorem{corollary}[theorem]{Corollary}
\newtheorem{proposition}[theorem]{Proposition}
\theoremstyle{remark}
\newtheorem{remark}[theorem]{Remark}
\newcommand{\HH}{\mathbb H}
\newcommand{\RR}{\mathbb R}
\newcommand{\SSS}{\mathbb S}
\newcommand{\dd}{\,\mathrm d}
\newcommand{\dmu}{\,\mathrm d\mu}
\newcommand{\nablab}{\overline\nabla}
\DeclareMathOperator{\tr}{tr}
\DeclareMathOperator{\Vol}{Vol}
\DeclareMathOperator{\diver}{div}
\title[Strongly stable CMC-one hypersurfaces in every dimension]{Strongly stable CMC-one hypersurfaces in every hyperbolic space of dimension at least four}
\author{Zihao Wang}
\address{School of Mathematical Sciences, Fudan University, Shanghai 200433, China}
\email{wangzh25@m.fudan.edu.cn}
\date{September 19, 2026}
\hypersetup{pdftitle={Strongly stable CMC-one hypersurfaces in every hyperbolic space of dimension at least four},pdfsubject={All-dimensional counterexamples to endpoint horospherical rigidity}}

\begin{document}
\begin{abstract}
For every integer $d\ge4$, we prove strong stability for a subfamily of classical rotational hypersurfaces in $\HH^d$ with normalized mean curvature one. The examples are complete, two-sided, properly embedded, and nowhere umbilic, with topology $\RR\times\SSS^{d-2}$. An explicit positive supersolution yields a quantitative stability inequality for all compactly supported test functions. Consequently, endpoint horospherical rigidity fails in every ambient dimension at least four.
\end{abstract}
\maketitle

\section{Introduction and main results}

We prove that nonhorospherical strongly stable hypersurfaces of normalized constant mean curvature one exist in every hyperbolic space of ambient dimension at least four. Throughout, $d$ denotes the \emph{ambient dimension}, so the hypersurfaces have dimension $d-1$. The hyperbolic metric has constant sectional curvature $-1$.

For an oriented hypersurface with global unit normal $N$, our conventions are
\[
 A(X)=-\nablab_XN,\qquad H=\frac{1}{d-1}\tr A,
 \qquad \Delta=\diver\nabla.
\]
A horosphere, suitably oriented, has $A=I$ and $H=1$. The quadratic form relevant to strong stability is
\begin{equation}\label{eq:Q}
 Q(\varphi)=\int_M\bigl(|\nabla\varphi|^2-(|A|^2-(d-1))\varphi^2\bigr)\dmu,
 \qquad \varphi\in C_c^\infty(M).
\end{equation}
A CMC hypersurface is \emph{strongly stable} if $Q(\varphi)\ge0$ for all such $\varphi$, whereas \emph{weak stability} tests only functions with $\int_M\varphi\dmu=0$. Thus strong stability implies weak stability. The form \eqref{eq:Q} is the second variation of area with its volume Lagrange multiplier; it does not assert that a nonminimal hypersurface is an unconstrained critical point of area. These conventions agree, after normalizing $H$, with \cite[Introduction]{Miranda}.

\subsection{Stability and the surface case}
The variational theory of CMC hypersurfaces was developed in the work of Barbosa and do Carmo \cite{BarbosaCarmo} and Barbosa, do Carmo and Eschenburg \cite{BarbosaCarmoEschenburg}. In particular, volume-preserving stability leads to spherical rigidity for closed CMC hypersurfaces in Euclidean space. For noncompact hypersurfaces, the admissible variations and the behavior at infinity become essential. At the analytic level, the relation between stability and positive solutions of a Schr\"odinger equation is central to the work of Fischer-Colbrie and Schoen \cite{FischerColbrieSchoen}. Our proof uses an elementary positive-supersolution identity, which we include below.

In $\HH^3$, da Silveira \cite{Silveira} proved that a complete noncompact strongly stable CMC surface with $|H|\ge1$ is a horosphere. This rigidity contrasts with the abundance of complete CMC-one surfaces without the stability assumption. Bryant \cite{Bryant} introduced a holomorphic representation for such surfaces, and Umehara and Yamada \cite{UmeharaYamada} developed the construction of complete examples. Do Carmo and da Silveira \cite{CarmoSilveira} related finite total curvature to finite index; de Lima and Rossman \cite{LimaRossman} studied the index through the underlying Riemann surface and secondary Gauss map, including explicit calculations for catenoid cousins. Collin, Hauswirth and Rosenberg \cite{CollinHauswirthRosenberg} proved that properly embedded CMC-one surfaces of finite topology have finite total curvature and regular ends, and identified the annular examples as catenoid cousins. These results distinguish finite index and controlled end geometry from global stability.

\subsection{Rotational geometry, ends and curvature estimates}
Do Carmo and Dajczer \cite{doCarmoDajczer} systematically studied rotation hypersurfaces in space forms. The precise family used here is already contained in Perdomo's construction \cite[\S7.1 and Theorem~7.1 in the arXiv version]{Perdomo}; the parameter correspondence is given in Section~\ref{subsec:classical}. Thus the construction of the embeddings is not the new issue in this paper.

Stability questions for rotational examples have also been studied in other regimes. B\'erard and Sa Earp \cite{BerardSaEarp} investigated maximal stable domains on catenoids and stable domains on CMC-one catenoid cousins in $\HH^3$. Stability of a proper subdomain is not the same as stability of the complete surface. B\'erard, de Lima and Rossman \cite{BerardLimaRossman} obtained index-growth results for Delaunay hypersurfaces, including hyperbolic examples with $H>1$; this strict-inequality regime is distinct from the endpoint considered here.

Global restrictions on ends likewise depend on the mean-curvature range. Cheng, Cheung and Zhou \cite{ChengCheungZhou} used harmonic-function methods to prove one-end theorems for weakly stable CMC hypersurfaces. Their hyperbolic conclusions include $\HH^4$ when $H^2\ge10/9$ and $\HH^5$ when $H^2\ge7/4$, and do not cover $H=1$. In ambient dimension three, Rosenberg, Souam and Toubiana \cite{RosenbergSouamToubiana} established curvature estimates for stable $H$-surfaces under sectional-curvature bounds. These results supply important structural context, but do not determine the stability of the complete higher-dimensional rotational hypersurfaces treated below.

\subsection{Related rigidity questions and the endpoint}
A related question of do Carmo asks whether a complete noncompact stable CMC hypersurface in Euclidean space must be minimal. For the strong-stability formulation and finite-index extensions, see Nelli and Pontuale \cite{NelliPontuale}. Their results include minimality under subexponential volume growth in every dimension, and exclude complete noncompact finite-index hypersurfaces with the same growth condition in $\HH^d$ when $H>1$ \cite[Theorem~2.2]{NelliPontuale}. The distinction between minimality and flatness, and between $H>1$ and $H=1$, is important here.

Chodosh's ICM report \cite[\S13, final list item~(2); arXiv version, p.~19]{Chodosh} records the hyperbolic rigidity question for two-sided stable CMC immersions with unnormalized mean curvature of absolute value at least the hypersurface dimension. Miranda \cite[Introduction]{Miranda} explicitly states its complete, strongly stable formulation. In our notation, its noncompact version asks:
\begin{quote}
\emph{Must a complete, connected, noncompact, two-sided strongly stable CMC immersion $M^{d-1}\to\HH^d$ with $|H|\ge1$ be a horosphere?}
\end{quote}
Hong \cite[Introduction; arXiv version~4, p.~4]{Hong} separately asks whether a complete noncompact CMC-one hypersurface in $\HH^4$, stable even in the weak sense, must be a horosphere. The two-dimensional theorem of da Silveira gives the affirmative surface case.

The strict inequality has a different history. Hong \cite[Theorem~1.1]{Hong} obtains nonexistence of complete noncompact finite-index CMC hypersurfaces in $\HH^4$ with $H>1$, assuming finitely many ends and finite-dimensional $H^1_0(M)$, the compactly supported first de Rham cohomology. Miranda \cite[Theorem~D and Corollary~E]{Miranda} proves related compactness and weak-stability rigidity results in ambient dimension six under a larger mean-curvature threshold; in $\HH^6$, that threshold is $|H|>7/5$ in our normalization. Neither result covers the endpoint addressed here.

Moreover, the broad horospherical statement is already false in higher dimension with $H>1$: Wang \cite[Proposition~7.2 and Corollary~7.3]{Wang} gives a strongly stable tube in $\HH^7$ with underlying manifold $\HH^3\times\SSS^3$ and normalized $H^2=49/48$. Our result concerns instead the \emph{exact endpoint} $H=1$, already in $\HH^4$ and in \emph{every} higher ambient dimension, for a two-ended rotational family. The examples have polynomial volume growth and bounded second fundamental form, so those properties alone do not restore endpoint horospherical rigidity.

\subsection{Main results}
The following theorem gives an explicit sufficient criterion for strong stability and a weighted inequality for arbitrary compactly supported test functions.

\begin{theorem}[Rotational strong-stability criterion]\label{thm:family}
Let $d\ge4$ be an integer and let $a>0$. Define
\begin{equation}\label{eq:parameters}
 c=a^{d-2}\bigl(\sqrt{1+a^2}-a\bigr),
 \qquad \alpha=\frac{c}{a^{d-1}}=\sqrt{1+a^{-2}}-1.
\end{equation}
Let $r$ solve
\begin{equation}\label{eq:ode}
 r''=(d-3)cr^{2-d}+(d-2)c^2r^{3-2d},
 \qquad r(0)=a,\qquad r'(0)=0.
\end{equation}
Then $r$ is smooth and even on $\RR$, satisfies $r\ge a$, and tends to infinity at both ends. Set
\begin{equation}\label{eq:Rt}
 R=\sqrt{1+r^2},\qquad v=r+cr^{2-d},
 \qquad t(s)=\int_0^s\frac{v(\tau)}{1+r(\tau)^2}\dd\tau.
\end{equation}
In the hyperboloid model, the map
\begin{equation}\label{eq:embedding}
 F_{d,a}(s,\omega)=\bigl(R(s)\cosh t(s),R(s)\sinh t(s),r(s)\omega\bigr),
 \quad (s,\omega)\in\RR\times\SSS^{d-2},
\end{equation}
is a complete, connected, boundaryless, two-sided, proper embedding into $\HH^d$ with $H=1$. Its induced metric is
\begin{equation}\label{eq:metric}
 g=\mathrm ds^2+r^2g_{\SSS^{d-2}},
\end{equation}
and its principal curvatures are
\begin{equation}\label{eq:curvatures}
 \lambda=1+\frac{c}{r^{d-1}}\quad(d-2\text{ times}),
 \qquad \mu=1-(d-2)\frac{c}{r^{d-1}}.
\end{equation}
In particular, the hypersurface is nowhere umbilic. If
\begin{equation}\label{eq:threshold}
 a\ge a_*(d):=\frac{2(d-2)}{(d-1)(d-3)},
\end{equation}
then it is strongly stable. More precisely, writing
\begin{equation}\label{eq:beta}
 \beta=(d-3)^2-2(d-2)\alpha\ge0,
\end{equation}
one has
\begin{equation}\label{eq:weighted}
 Q(\varphi)\ge
 \int_M r^{6-2d}\bigl|\nabla(r^{d-3}\varphi)\bigr|^2\dmu
 +\beta c\int_M r^{1-d}\varphi^2\dmu
 \qquad\text{for all }\varphi\in C_c^\infty(M).
\end{equation}
The threshold \eqref{eq:threshold} is sufficient; no optimality is asserted.
\end{theorem}

\begin{corollary}[Explicit counterexamples in every $\HH^d$, $d\ge4$]\label{cor:all}
For each integer $d\ge4$, take
\begin{equation}\label{eq:uniform}
 a=\frac{12}{5},\qquad c_d=\frac{12^{d-2}}{5^{d-1}},
 \qquad \beta_d=(d-3)^2-\frac{d-2}{6}.
\end{equation}
Then $M_d:=F_{d,12/5}(\RR\times\SSS^{d-2})\subset\HH^d$ is a complete, connected, noncompact, boundaryless, two-sided, strongly stable and properly embedded CMC hypersurface with $H=1$ that is not a horosphere. For every $\varphi\in C_c^\infty(M_d)$,
\begin{equation}\label{eq:explicit-weight}
 Q(\varphi)\ge
 \int_{M_d}r^{6-2d}\bigl|\nabla(r^{d-3}\varphi)\bigr|^2\dmu
 +\beta_dc_d\int_{M_d}r^{1-d}\varphi^2\dmu,
 \qquad \beta_d\ge\frac23.
\end{equation}
Thus the endpoint horospherical conclusion fails in every ambient dimension $d\ge4$, whether stability is interpreted in the strong or the weak sense.
\end{corollary}

The higher-dimensional examples are constructed directly as hypersurfaces of $\HH^d$. They are not obtained by viewing a hypersurface of $\HH^4$ inside a higher-codimension totally geodesic subspace, nor by taking a Riemannian product. The profile equation, the spherical multiplicity and the Jacobi potential all vary with $d$.

\subsection{Relation to the classical rotational family}\label{subsec:classical}
Rotational hypersurfaces in space forms were systematically studied by do Carmo and Dajczer \cite{doCarmoDajczer}. The precise embeddings used here appear in Perdomo \cite[\S7.1 and Theorem~7.1 in the arXiv version]{Perdomo}. In his notation, set $n=d-1$ and $H=1$. The profile $g$ and parameter $C>0$ satisfy
\begin{equation}\label{eq:perdomo}
 (g')^2+g^{2-2n}+2g^{2-n}=C,\qquad r=\frac{g}{\sqrt C}.
\end{equation}
The substitution $C=c^{-2/n}$ and $g=\sqrt C\,r$ transforms \eqref{eq:perdomo} into the first integral \eqref{eq:first} below. His principal curvatures become \eqref{eq:curvatures}, and his profile parameter $\theta$ is our $t$. Up to reordering the Lorentzian coordinates, his embedding is exactly \eqref{eq:embedding}. We therefore make no claim that the rotational embeddings themselves are new. The contribution established here is the explicit strong-stability criterion and its consequence for endpoint horospherical rigidity. The geometric and stability calculations needed for the result are included in full. Section~2 establishes the global embedding and its principal curvatures. Section~3 proves the stability criterion. Section~4 gives uniform explicit parameters and a second computation of the full quadratic form. Section~5 describes the ends and explains the failure of strong stability in the surface case.

\section{The complete rotational embedding}\label{sec:geometry}

\subsection{The global profile and its first integral}
The initial-value problem \eqref{eq:ode} has a unique local smooth solution as long as $r>0$. Direct differentiation gives
\[
 \frac{\mathrm d}{\mathrm ds}
 \bigl((r')^2+2cr^{3-d}+c^2r^{4-2d}\bigr)=0.
\]
Writing $b=\sqrt{1+a^2}-a$, we have $b^2+2ab=1$ and $c=a^{d-2}b$. Evaluation at $s=0$ therefore yields
\begin{equation}\label{eq:first}
 (r')^2=1-2cr^{3-d}-c^2r^{4-2d}.
\end{equation}
Since $d\ge4$ and $c>0$, equation \eqref{eq:ode} implies $r''>0$. On the positive half of the maximal existence interval, $r'>0$ and $r\ge a$. Equation \eqref{eq:first} gives $r'<1$. Thus, on any finite interval $0\le s\le S$ on which the solution is defined,
\[
 a\le r(s)\le a+S,\qquad |r'(s)|\le1.
\]
The smooth ODE cannot cease to exist at a finite positive time. Uniqueness and invariance under $s\mapsto-s$ give a global even solution. In particular, smoothness across the neck does not rely on gluing solutions of a first-order square-root equation.

For any $s_0>0$, strict convexity gives $r'(s)\ge r'(s_0)>0$ for $s\ge s_0$. Hence $r(s)\to\infty$ as $s\to\infty$, and also as $s\to-\infty$ by evenness. Equation \eqref{eq:first} then implies
\begin{equation}\label{eq:asymptotics}
 \lim_{s\to\infty}r'(s)=1,
 \qquad \lim_{|s|\to\infty}\frac{r(s)}{|s|}=1.
\end{equation}
With $p=r'$ and $v$ as in \eqref{eq:Rt}, we will repeatedly use
\begin{equation}\label{eq:identities}
 p^2+v^2=R^2,\qquad r''=r-vv_r,
 \qquad v_r=1-(d-2)cr^{1-d}.
\end{equation}
Here $v_r$ denotes differentiation with respect to $r$, not with respect to $s$.

\subsection{Metric, normal and properness}
Write an ambient point as $X=(X_0,X_1,Z)$, with $Z\in\RR^{d-1}$, and set
\[
 \langle X,X\rangle_L=-X_0^2+X_1^2+|Z|^2,
 \qquad \HH^d=\{\langle X,X\rangle_L=-1,\ X_0>0\}.
\]
The map \eqref{eq:embedding} takes values in $\HH^d$ because $-R^2+r^2=-1$. A direct differentiation, using \eqref{eq:identities} and $t'=v/R^2$, gives
\[
 F_{d,a}^*g_{\HH^d}
 =\left(\frac{p^2}{R^2}+R^2(t')^2\right)\mathrm ds^2
   +r^2g_{\SSS^{d-2}}
 =\mathrm ds^2+r^2g_{\SSS^{d-2}}.
\]
This positive-definite metric dominates the complete product metric $\mathrm ds^2+a^2g_{\SSS^{d-2}}$, and is therefore complete.

Along the immersion, define
\begin{equation}\label{eq:frame}
 e_\rho=(r\cosh t,r\sinh t,R\omega),
 \qquad e_t=(\sinh t,\cosh t,0).
\end{equation}
These are orthogonal unit tangent vectors to $\HH^d$. The meridional unit tangent and a global unit normal are
\begin{equation}\label{eq:normal}
 T=F_s=\frac pR e_\rho+\frac vR e_t,
 \qquad N=-\frac vR e_\rho+\frac pR e_t.
\end{equation}
The normal is orthogonal to the sphere directions, and \eqref{eq:identities} gives $|N|=1$. It is smooth everywhere, including at $s=0$, so the immersion is two-sided.

Since $v>0$, the function $t$ is strictly increasing. The ratio $X_1/X_0=\tanh t(s)$ at an image point uniquely determines $s$, and then $Z/r(s)$ determines $\omega$. Thus $F_{d,a}$ is injective. Furthermore,
\[
 X_0=R\cosh t\ge R\longrightarrow\infty
 \qquad\text{as }|s|\longrightarrow\infty.
\]
The inverse image of a compact subset of $\HH^d$ is therefore closed and contained in a compact slab of $\RR\times\SSS^{d-2}$. Hence $F_{d,a}$ is proper. A proper injective immersion is an embedding. The domain is connected and without boundary.

\subsection{Principal curvatures, including at the neck}
We compute the shape operator without dividing by $p=r'$. Put $r=\sinh\rho$ and $R=\cosh\rho$. In these coordinates the ambient metric is
\[
 g_{\HH^d}=\mathrm d\rho^2+\cosh^2\rho\,\mathrm dt^2
             +\sinh^2\rho\,g_{\SSS^{d-2}}.
\]
Thus $r$ is the radius of the spherical orbit, whereas $\rho=\operatorname{arsinh}r$ is the ambient distance to the rotation axis. The connection identities are
\begin{equation}\label{eq:connection}
 \nablab_{e_\rho}e_\rho=0,\quad
 \nablab_{e_\rho}e_t=0,\quad
 \nablab_{e_t}e_\rho=\frac rR e_t,\quad
 \nablab_{e_t}e_t=-\frac rR e_\rho.
\end{equation}
For a unit sphere tangent $E$, one also has $\nablab_Ee_\rho=(R/r)E$ and $\nablab_Ee_t=0$. Since the coefficients of $N$ depend only on $s$,
\[
 -\nablab_EN=\frac vr E.
\]
The spherical principal curvature is therefore $\lambda=v/r$, with multiplicity $d-2$.

Set $\xi=p/R$ and $\eta=v/R$, so that $N=-\eta e_\rho+\xi e_t$. Along the meridian,
\[
 \xi'=\frac{r''}{R}-\frac{rp^2}{R^3},
 \qquad \eta'=\frac{v_rp}{R}-\frac{rpv}{R^3}.
\]
Using \eqref{eq:connection} and then \eqref{eq:identities}, we obtain
\begin{align*}
 \nablab_TN
 &=\left(-\eta'-\xi\eta\frac rR\right)e_\rho
   +\left(\xi'-\eta^2\frac rR\right)e_t\\
 &=-\frac{v_rp}{R}e_\rho+\frac{r''-r}{R}e_t
 =-v_rT.
\end{align*}
Thus the meridional principal curvature is $\mu=v_r$. All expressions are smooth at $p=0$, so the calculation is valid at the neck as well. This proves \eqref{eq:curvatures}, and hence
\begin{equation}\label{eq:potential}
 (d-2)\lambda+\mu=d-1,
 \qquad q:=|A|^2-(d-1)=|A-I|^2
       =(d-1)(d-2)c^2r^{2-2d}>0.
\end{equation}
In particular, $H=1$ and $\lambda-\mu=(d-1)cr^{1-d}>0$ at every point. This completes the geometric part of Theorem~\ref{thm:family}.

\section{Strong stability in all dimensions}\label{sec:stability}

\subsection{An explicit positive supersolution}
For the warped metric \eqref{eq:metric}, every smooth radial function $f=f(s)$ satisfies
\begin{equation}\label{eq:radiallap}
 \Delta f=f''+(d-2)\frac{r'}r f'.
\end{equation}
Let $L=\Delta+q$ be the Jacobi operator, and take
\begin{equation}\label{eq:supersolution}
 u=r^{3-d}>0.
\end{equation}
Substituting into \eqref{eq:radiallap}, the terms containing $(r')^2$ cancel exactly:
\begin{equation}\label{eq:lapu}
 \frac{\Delta u}{u}=-(d-3)\frac{r''}r.
\end{equation}
Combining \eqref{eq:ode} and \eqref{eq:potential} gives
\begin{equation}\label{eq:Lu}
 \frac{Lu}{u}=-(d-3)^2cr^{1-d}+2(d-2)c^2r^{2-2d}.
\end{equation}
Because $r\ge a$,
\begin{equation}\label{eq:Lbound}
 -\frac{Lu}{u}
 =\frac{c}{r^{d-1}}\left((d-3)^2-2(d-2)\frac{c}{r^{d-1}}\right)
 \ge\beta cr^{1-d}.
\end{equation}
The condition $\beta\ge0$ is equivalent to \eqref{eq:threshold}. Indeed, with $B=(d-3)^2/(2(d-2))>0$,
\begin{align*}
 \beta\ge0
 &\Longleftrightarrow \sqrt{1+a^{-2}}-1\le B\\
 &\Longleftrightarrow a^{-2}\le B(B+2)
   =\frac{(d-1)^2(d-3)^2}{4(d-2)^2}\\
 &\Longleftrightarrow a\ge\frac{2(d-2)}{(d-1)(d-3)}.
\end{align*}
All quantities squared in this equivalence are positive.

\subsection{The full quadratic form}
For any positive smooth function $u$ and any $\varphi\in C_c^\infty(M)$, write $\zeta=\varphi/u$. The pointwise identity
\[
 |\nabla(u\zeta)|^2
 =u^2|\nabla\zeta|^2+\langle\nabla u,\nabla(u\zeta^2)\rangle
\]
and integration by parts yield
\begin{equation}\label{eq:groundstate}
 Q(\varphi)=\int_Mu^2\left|\nabla\left(\frac\varphi u\right)\right|^2\dmu
            -\int_M\frac{Lu}{u}\varphi^2\dmu.
\end{equation}
The integration by parts is legitimate because $u\zeta^2=\varphi^2/u$ is compactly supported. No integrability assumption on $u$ at infinity is needed.

With $u=r^{3-d}$, equations \eqref{eq:Lu} and \eqref{eq:groundstate} give the exact identity
\begin{align}\label{eq:exact-groundstate}
 Q(\varphi)
 &=\int_Mr^{6-2d}\bigl|\nabla(r^{d-3}\varphi)\bigr|^2\dmu\notag\\
 &\quad+\int_M\left((d-3)^2cr^{1-d}-2(d-2)c^2r^{2-2d}\right)\varphi^2\dmu.
\end{align}
Applying \eqref{eq:Lbound} proves \eqref{eq:weighted} and completes the proof of Theorem~\ref{thm:family}. The argument applies to arbitrary functions of $(s,\omega)$: it imposes neither rotational symmetry nor the zero-integral restriction of weak stability.

\begin{remark}\label{rem:nogap}
If $a>a_*(d)$, the coefficient $\beta$ is strictly positive. Nevertheless, the weight $r^{1-d}$ decays at infinity. The weighted inequality is not a claim of a uniform positive lower bound for $Q$ by a multiple of $\|\varphi\|_{L^2(M)}^2$.
\end{remark}

\section{One explicit neck radius for every ambient dimension}\label{sec:explicit}

\subsection{Uniform rational parameters}
Take $a=12/5$ in every dimension. Since $\sqrt{1+a^2}=13/5$, equation \eqref{eq:parameters} gives
\[
 c=c_d=\frac15\left(\frac{12}{5}\right)^{d-2}
      =\frac{12^{d-2}}{5^{d-1}},
 \qquad \alpha=\frac1{12}.
\]
Consequently,
\begin{equation}\label{eq:beta-positive}
 \beta_d=(d-3)^2-\frac{d-2}{6}
 =\frac23+\frac{(d-4)(6d-13)}6\ge\frac23
 \qquad(d\ge4).
\end{equation}
Theorem~\ref{thm:family} now proves Corollary~\ref{cor:all}. The same neck radius works in all dimensions, although the profile $r=r_d$ and the constant $c_d$ depend on $d$.

At the neck the principal curvatures and potential are
\begin{equation}\label{eq:neckvalues}
 \lambda=\frac{13}{12}\quad(d-2\text{ times}),
 \qquad \mu=\frac{14-d}{12},
 \qquad q=\frac{(d-1)(d-2)}{144}.
\end{equation}
Their normalized mean is exactly one, and their nonumbilicity is explicit. The sign of the meridional curvature is immaterial to the stability proof; no convexity assumption is used.

\subsection{An all-dimensional conjugation retaining angular derivatives}
There is a second direct verification of strong stability for this explicit family. It also generalizes the one-dimensional conjugation in the case of a three-dimensional hypersurface. Let
\[
 m=\frac{d-2}{2},\qquad \psi=r^m\varphi.
\]
Since $\mathrm d\mu=r^{d-2}\mathrm ds\,\mathrm d\omega$, substituting $\varphi=r^{-m}\psi$ into \eqref{eq:Q} gives
\begin{equation}\label{eq:conjugated}
 Q(\varphi)=\int_{\RR\times\SSS^{d-2}}
 \left(\psi_s^2+\frac{|\nabla_{\SSS^{d-2}}\psi|^2}{r^2}
             +V_d\psi^2\right)\dd s\dd\omega,
\end{equation}
where
\begin{align}\label{eq:V}
 V_d
 &=m\frac{r''}r+m(m-1)\left(\frac{r'}r\right)^2-q\notag\\
 &=\frac{(d-2)(d-4)}4\left(\frac{r'}r\right)^2
   +\frac{(d-2)c_d}{2r^{d-1}}
       \left(d-3-d\frac{c_d}{r^{d-1}}\right).
\end{align}
For clarity, the cross term from the radial derivative is $-2m(r'/r)\psi\psi_s$. Integration by parts replaces it by $m(r'/r)'\psi^2$. Combining this with $m^2(r'/r)^2\psi^2$ gives the first line of \eqref{eq:V}; the second line follows from \eqref{eq:ode} and \eqref{eq:potential}. Compact support eliminates all boundary terms.

For $d\ge4$, the first term in the second line of \eqref{eq:V} is nonnegative. Moreover, $c_d/r^{d-1}\le1/12$, so
\begin{equation}\label{eq:Vpositive}
 V_d\ge\delta_dc_dr^{1-d},
 \qquad \delta_d:=\frac{(d-2)(11d-36)}{24}>0.
\end{equation}
Equations \eqref{eq:conjugated} and \eqref{eq:Vpositive} give a separate verification of strong stability in every $d\ge4$. In particular, all angular derivatives remain present in the calculation; no spherical-harmonic restriction is made.

\subsection{Recovery of the example in \texorpdfstring{$\HH^4$}{H4}}
When $d=4$, the hypersurface has dimension three and
\[
 a=\frac{12}{5},\qquad c_4=\frac{144}{125},
 \qquad \beta_4=\delta_4=\frac23.
\]
Here $\psi=r\varphi$ and $V_4=c_4r^{-3}(1-4c_4r^{-3})$. Thus either stability computation yields
\begin{equation}\label{eq:d4}
 Q(\varphi)\ge\frac{96}{125}\int_{M_4}r^{-3}\varphi^2\dmu.
\end{equation}
At the neck, the principal curvatures are $13/12,13/12,5/6$, and $q=1/24$. The four-dimensional ambient example is therefore a particular member of the all-dimensional family, not a separate construction.

\section{Geometry of the ends and the dimensional endpoint}\label{sec:ends}

\begin{proposition}\label{prop:ends}
For every $d\ge4$ and every $a>0$, the hypersurface of Theorem~\ref{thm:family} has two ends, finite topology and bounded second fundamental form. For $o=(0,\omega_0)$ and intrinsic balls $B_\ell(o)$,
\begin{equation}\label{eq:volume}
 \lim_{\ell\to\infty}\frac{\Vol(B_\ell(o))}{\ell^{d-1}}=2\omega_{d-1},
\end{equation}
where $\omega_{d-1}$ is the Euclidean unit-ball volume in dimension $d-1$. Moreover, for every $p>0$,
\begin{equation}\label{eq:total}
 \int_M|A-I|^p\dmu<\infty\quad\Longleftrightarrow\quad p>1.
\end{equation}
In particular, the strongly stable examples have finite $\int_M|A-I|^{d-1}\dmu$.
\end{proposition}

\begin{proof}
The topology is $\RR\times\SSS^{d-2}$, so the complement of a sufficiently large compact slab has exactly two connected components. Boundedness of $A$ follows from $r\ge a$ and \eqref{eq:curvatures}.

The intrinsic distance from $o$ to $(s,\omega)$ is at least $|s|$. It is at most $|s|+\pi a$: first travel along the neck sphere and then along a meridian. Hence, for $\ell>\pi a$,
\[
 \{|s|<\ell-\pi a\}\subset B_\ell(o)\subset\{|s|<\ell\}.
\]
The slab $\{|s|<\ell\}$ has volume
\[
 2|\SSS^{d-2}|\int_0^\ell r(s)^{d-2}\dd s.
\]
The limit $r(s)/s\to1$ in \eqref{eq:asymptotics}, together with the two inclusions, proves \eqref{eq:volume}.

Finally, \eqref{eq:potential} gives
\[
 |A-I|=\sqrt{(d-1)(d-2)}\,c\,r^{1-d}.
\]
The integral in \eqref{eq:total} is therefore a positive constant times
\[
 \int_{\RR}r(s)^{d-2-(d-1)p}\dd s.
\]
There is no singularity on compact sets because $r\ge a>0$. By \eqref{eq:asymptotics}, the integral at infinity is finite exactly when $d-2-(d-1)p<-1$, or equivalently $p>1$.
\end{proof}

\begin{remark}[Why the family does not extend to stable surfaces in $\HH^3$]\label{rem:d3}
There is no conflict with da Silveira's surface rigidity theorem \cite{Silveira}. The obstruction is visible directly in the same rotational formulas at ambient dimension $d=3$. In that case,
\[
 0<c=a(\sqrt{1+a^2}-a)<\frac12,\qquad
 r''=c^2r^{-3},\qquad q=2c^2r^{-4}>0,
\]
and the induced metric is $\mathrm ds^2+r^2\mathrm d\theta^2$, with $r(s)\le a+|s|$. The even solution is global by the same ODE continuation argument.

For $T>1$, choose the Lipschitz cutoff
\[
 \chi_T(s)=
 \begin{cases}
 1,&|s|\le1,\\
 \log(T/|s|)/\log T,&1<|s|<T,\\
 0,&|s|\ge T.
 \end{cases}
\]
Its Dirichlet energy satisfies
\[
 \int_M|\nabla\chi_T|^2\dmu
 \le\frac{4\pi}{(\log T)^2}\int_1^T\frac{a+s}{s^2}\dd s
 =\frac{4\pi\bigl(a(1-T^{-1})+\log T\bigr)}{(\log T)^2}
 \longrightarrow0.
\]
On the other hand,
\[
 \int_M q\chi_T^2\dmu\ge\int_{\{|s|\le1\}}q\dmu>0.
\]
Thus $Q(\chi_T)<0$ for all sufficiently large $T$. Smooth compactly supported approximation preserves this strict inequality. The two-dimensional members of the rotational family are therefore not strongly stable.
\end{remark}

The counterexamples proved here concern the exact endpoint $H=1$. They do not assert the existence of stable noncompact hypersurfaces with $H>1$ in every dimension, and in particular do not contradict Hong's strict-inequality nonexistence theorem in $\HH^4$ under its additional hypotheses \cite{Hong}. Nor are these examples minimal hypersurfaces: the case $H=0$ is a different problem. Their role is to show, in every ambient dimension $d\ge4$, that endpoint strong stability alone does not force horospherical rigidity.

\section*{Acknowledgment of AI assistance}
The author used ChatGPT in this work.The author takes full responsibility for the content of the paper.

\end{document}